\documentclass{article}%
\usepackage{graphicx}
\usepackage{amsmath}
\usepackage{amsfonts}
\usepackage{amssymb}
\usepackage{setspace}
\usepackage[hmargin=3.5cm,vmargin=3.5cm]{geometry}%
\providecommand{\U}[1]{\protect\rule{.1in}{.1in}}

\newtheorem{theorem}{Theorem}
{}

{}
\newtheorem{notation}{Notation}

\newtheorem{proposition}{Proposition}
\newtheorem{remark}{Remark}

\newtheorem{summary}{Summary}
\newenvironment{proof}[1][Proof]{\textbf{#1.} }{\ \rule{0.5em}{0.5em}}

\begin{document}

\title{On the spectrum of the differential operators of even order with periodic
matrix coefficients}
\author{O. A. Veliev\\{\small Department of Mechanical Engineering, Dogus University, Istanbul,
Turkey.}\\{\small Email: oveliev@dogus.edu.tr}}
\date{}
\maketitle

\begin{abstract}
In this paper, we consider the band functions, Bloch functions and spectrum of
the self-adjoint differential operator $L$ with periodic matrix coefficients.
Conditions are found for the coefficients under which the number of gaps in
the spectrum of the operator $L$ is finite.



\end{abstract}

In this paper, we investigate the band functions, Bloch functions and spectrum
of the differential operator $L,$ generated in the space $L_{2}^{m}%
(\mathbb{R})$ of vector-valued functions by formally self-adjoint differential
expression
\begin{equation}
(-i)^{2\nu}y^{(2\nu)}(x)+%
{\textstyle\sum\limits_{k=2}^{2\nu-2}}
P_{k}(x)y^{(2\nu-k)}(x), \tag{1}%
\end{equation}
where $\nu>1$ and $P_{k}\left(  x\right)  ,$ for each $k=2,3,...2\nu-2,$ is
the $m\times m$ matrix with the summable entries $p_{k,i,j}$ satisfying
$p_{k,i,j}\left(  x+1\right)  =p_{k,i,j}\left(  x\right)  $ for all $i$
$=1,2,...m$ and $j$ $=1,2,...m.$

To explain the results of this paper, let us introduce some notations. It is
well-known that (see [1, 2, 4]) the spectrum $\sigma(L)$ of the operator $L$
is the union of the spectra $\sigma(L_{t})$ of the operators $L_{t},$ for
$t\in(-\pi,\pi],$ generated in $L_{2}^{m}[0,1]$ by (1) and the quasiperiodic
conditions
\begin{equation}
U_{p}(y):=y^{(p)}\left(  1\right)  -e^{it}y^{(p)}\left(  0\right)  =0,\text{
}p=0,1,...,(2\nu-1). \tag{2}%
\end{equation}
For $t\in(-\pi,\pi]$ the spectra $\sigma(L_{t})$ of the operators $L_{t}$
consist of the eigenvalues
\begin{equation}
\lambda_{1}(t)\leq\lambda_{2}(t)\leq\cdot\cdot\cdot\tag{3}%
\end{equation}
called the Bloch eigenvalues of $L$. The eigenfunctions $\Psi_{n,t}$
corresponding to the Bloch eigenvalues $\lambda_{n}(t)$ are the Bloch
functions of $L$.

In [10] the continuity of the band function $\lambda_{n}:t\rightarrow
\lambda_{n}(t)$ and Bloch function of the operator $L$ was investigated. In
Section 2, we improve these results as follows. We only assume that the
entries of the coefficients of (1) are summable function, while in [10] it was
assumed that they are bounded functions. In [10] the choice of the continuous
Bloch functions was made in a non-constructive way. Here we constructively
define the continuous Bloch functions. We prove that for each $f\in L_{2}%
^{m}[0,1]$ the function $P_{t}f(x)$ converges to $P_{a}f$ $(x)$ uniformly with
respect to $x\in\lbrack0,1]$ as $t\rightarrow a$, while in [10] this
convergence was done in the $L_{2}$ norm, where $P_{t}$ is the projection of
$L_{t}$ corresponding to the eigenvalue $\lambda_{n}(t).$ Moreover, the
methods used in Section 2 and [10] are completely different. Therefore,
Section 2 can be considered as a continuation and completion of the paper [10].

In Section 3, we consider the spectrum of the operator $L.$ Since $\sigma(L)$
is the union of $\sigma(L_{t})$ for $t\in(-\pi,\pi],$ \ the spectrum of $L$
consists of the sets
\begin{equation}
I_{n}=\left\{  \lambda_{n}(t):t\in(-\pi,\pi]\right\}  ,\tag{4}%
\end{equation}
for $n=1,2,....$ The set $I_{n}$ is called the $n$th band of the spectrum. The
band $I_{n}$ tends to infinity as $n\rightarrow\infty.$ The spaces between the
bands $I_{k}$ and $I_{k+1}$, for $k=1,2,...,$ are called the gaps in the
spectrum of $L.$ In Section 3, we prove that most of the positive real axis is
overlapped by $m$ bands of the spectrum and consider the gaps (see Theorems
5). Then we find a condition on the eigenvalues of the matrix
\begin{equation}
C=\int_{0}^{1}P_{2}\left(  x\right)  dx\tag{5}%
\end{equation}
for which the number of the gaps in the spectrum is finite (see Theorem 6).
Note that in [6], we proved Theorem 6 under the assumption that the matrix $C$
has three simple eigenvalues $\mu_{j_{1}},$ $\mu_{j_{2}}$ and $\mu_{j_{3}}$
satisfying (30). In this paper, we prove Theorem 6 without any conditions on
the multiplicity of these eigenvalues. The case $\nu=1$ was investigated in
[9]. Parts of the proofs of Theorems 5 and 6, similar to the proofs of the
case $\nu=1,$ are omitted and references to [9] are given.

\section{On the band functions and Bloch functions}

\ \ In this section, first, we study the continuity of the band functions and
Bloch functions of $L$ with respect to the quasimomentum by using the
following well-known statements (see for example [5] Chap. 3) formulated here
as summary.

\begin{summary}
The eigenvalues of $L_{t}$ are the roots of the characteristic determinant
\begin{equation}
\Delta(\lambda,t)=\det(Y_{j}^{(p-1)}(1,\lambda)-e^{it}Y_{j}^{(p-1)}%
(0,\lambda))_{j,p=1}^{2\nu}= \tag{6}%
\end{equation}%
\[
e^{i2\nu mt}+f_{1}(\lambda)e^{i(2\nu m-1)t}+f_{2}(\lambda)e^{i(2\nu
m-2)t}+...+f_{2\nu m-1}(\lambda)e^{it}+1
\]
which is a polynomial of $e^{it}$\ with entire coefficients $f_{1}%
(\lambda),f_{2}(\lambda),...$, where

$Y_{1}(x,\lambda),Y_{2}(x,\lambda),\ldots,Y_{2\nu}(x,\lambda)$ are the
solutions of the matrix equation
\[
(-i)^{2\nu}Y^{(2\nu)}+P_{2}Y^{(2\nu-2)}+P_{3}Y^{(2\nu-3)}+...+P_{2\nu
}Y=\lambda Y,
\]
satisfying $Y_{k}^{(j)}(0,\lambda)=O$ for $j\neq k-1$ and $Y_{k}%
^{(k-1)}(0,\lambda)=I$. Here, $O$ and $I$ are $m\times m$ zero and identity
matrices, respectively. The Green's function of $L_{t}-\lambda I$ is defined
by formula%
\begin{equation}
G(x,\xi,\lambda,t)=g(x,\xi,\lambda)-\frac{1}{\Delta(\lambda,t)}\sum
\limits_{j,p=1}^{2\nu}Y_{j}(x,\lambda)V_{jp}(x,\lambda)U_{p}(g),\tag{7}%
\end{equation}
where $g$ does not depend on $t$ and $V_{jp}$ is the transpose of that $m$th
order matrix consisting of the cofactor of the element $U_{p}(Y_{j})$ in the
determinant $\det(U_{p}(Y_{j}))_{j,p=1}^{2\nu}.$ Hence, the entries of the
matrices $V_{jp}(x,\lambda)$ and $U_{p}(g)$ either do not depend on $t$ or
have the forms $u^{(p)}(1,\lambda)-e^{it}u^{(p)}(0,\lambda)$ and
$h(1,\xi,\lambda)-e^{it}h(0,\xi,\lambda)$ respectively, where the functions
$u$ and $h$ do not depend on $t.$
\end{summary}

Now, using this summary, we prove that for each $n$\ the function $\lambda
_{n}$ defined in (3) is continuous at each point $a\in(-\pi,\pi].$ For this we
introduce the following notations.

\begin{notation}
Let $\Lambda_{1}(a)<\Lambda_{2}(a)<\cdot\cdot\cdot$ be the distinct
eigenvalues of $L_{a}$ with the multiplicities $k_{1},k_{2},...,$
respectively. For each $n$ there exists $p$ such that $n\leq k_{1}+k_{2}%
+\cdot\cdot\cdot+k_{p}.$ These notations with the notation (3) imply that
$\lambda_{s_{j-1}+1}(a)=\lambda_{s_{j-1}+2}(a)=\cdot\cdot\cdot=\lambda_{s_{j}%
}(a)=\Lambda_{j}(a),$ where $s_{0}=0,$ $s_{j}=k_{1}+k_{2}+\cdot\cdot
\cdot+k_{j}$ for $j=1,2,...,p$ and $n\leq s_{p}.$ Since $L$ is below-bounded
operator, there exists $b\in\mathbb{R}$ such that $\lambda_{1}(t)>b$ for all
$t\in(-\pi,\pi].$
\end{notation}

Now, we are ready to prove the following theorem.

\begin{theorem}
$(a)$ For every $r>0$ satisfying the inequality
\begin{equation}
r<\tfrac{1}{2}\min_{j=1,2,...p}\left(  \Lambda_{j+1}(a)-\Lambda_{j}(a)\right)
\tag{8}%
\end{equation}
there exists $\delta>0$ such that the operator $L_{t}$, for $t\in
(a-\delta,a+\delta),$ has \ $k_{j}$ eigenvalues in the interval $\left(
\Lambda_{j}(a)-r,\Lambda_{j}(a)+r\right)  $, where $j=1,2,...,p$ and
$a\in(-\pi,\pi].$

$(b)$ The eigenvalues of $L_{t}$ for $t\in(a-\delta,a+\delta)$ lying in
$\left(  \Lambda_{j}(a)-r,\Lambda_{j}(a)+r\right)  $ are $\lambda_{s_{j-1}%
+1}(t),\lambda_{s_{j-1}+2}(t),...,\lambda_{s_{j}}(t),$ where $s_{j}$ is
defined in Notation 1.
\end{theorem}

\begin{proof}
$(a)$ By (8), the circle $D_{j}(a)=\left\{  z\in\mathbb{C}:\left\vert
z-\Lambda_{j}(a)\right\vert =r\right\}  $ belongs to the resolvent set of the
operator $L_{a}.$ It means that $\Delta(\lambda,a)\neq0$ for each $\lambda\in
D_{j}(a).$ Since $\Delta(\lambda,a)$ is a continuous function on the compact
$D_{j}(a),$ there exists $c>0$ such that $\left\vert \Delta(\lambda
,a)\right\vert >c$ for all $\lambda\in D_{j}(a).$ Moreover, by (6),
$\Delta(\lambda,t)$ is a polynomial of $e^{it}$ with entire coefficients.
Therefore, there exists $\delta>0$ such that
\begin{equation}
\left\vert \Delta(\lambda,t)\right\vert >c/2\tag{9}%
\end{equation}
for all $t\in(a-\delta,a+\delta)$ and $\lambda\in D_{j}(a).$ It implies that
$D_{j}(a)$ belongs to the resolvent set of $L_{t}$ for all $t\in
(a-\delta,a+\delta).$ On the other hand, it is well-known that
\begin{equation}
\left(  L_{t}-\lambda I\right)  ^{-1}f(x)=\int_{0}^{1}G(x,\xi,\lambda
,t)f(\xi)d\xi,\tag{10}%
\end{equation}
where $G(x,\xi,\lambda,t)$ is the Green's function of $L_{t}$ defined in (7).
Moreover, it easily follows from Summary 1 and (9) that there exists $M$ such
that
\begin{equation}
\left\vert G(x,\xi,\lambda,t)-G(x,\xi,\lambda,a)\right\vert \leq M\left\vert
t-a\right\vert \tag{11}%
\end{equation}
for all $x\in\lbrack0,1],$ $\xi\in\lbrack0,1]$, $\lambda\in D_{j}(a)$ and
$t\in(a-\delta,a+\delta).$ Therefore, using (10) and Summary 1, one can easily
verify that $\left(  L_{t}-\lambda I\right)  ^{-1}$ for $\lambda\in D_{j}(a)$
and the projection
\begin{equation}
P_{t}=%
{\textstyle\int\nolimits_{D_{j}(a)}}
\left(  L_{t}-\lambda I\right)  ^{-1}d\lambda\tag{12}%
\end{equation}
continuously depend on $t\in(a-\delta,a+\delta).$ This implies that the
operator $L_{t}$ for each $t\in(a-\delta,a+\delta)$ has \ $k_{j}$ eigenvalues
inside $D_{j}(a)$ and, therefore, in the interval $\left(  \Lambda
_{j}(a)-r,\Lambda_{j}(a)-r\right)  ,$ since $L_{a}$ has \ $k_{j}$ eigenvalues
(counting multiplicity) inside $D_{j}(a)$.

$(b)$ Since $L_{a}$ has no eigenvalues in the intervals $[b,\Lambda_{1}(a)-r]$
and $\left[  \Lambda_{j}(a)+r,\Lambda_{j+1}(a)-r\right]  $ for $j=1,2,...,p$,
arguing as above we obtain that $L_{t}$ for $t\in(a-\delta,a+\delta)$ also has
no eigenvalues in these closed intervals. Therefore, the eigenvalues of
$L_{t}$ for $t\in(a-\delta,a+\delta)$ lying in $\left(  \Lambda_{j}%
(a)-r,\Lambda_{j}(a)+r\right)  $ are $\lambda_{s_{j-1}+1}(t),\lambda
_{s_{j-1}+2}(t),...,\lambda_{s_{j}}(t)$ for $j=1,2,...,p.$
\end{proof}

Now, using these statements, we prove the main results of this section.

\begin{theorem}
$(a)$ For each $n$\ the function $\lambda_{n}$ defined in (3) is continuous at
$a\in(-\pi,\pi].$

$(b)$ For each $f\in L_{2}^{m}[0,1]$ we have $\left\Vert P_{t}f-P_{a}%
f\right\Vert _{\infty}\rightarrow0$ as $t\rightarrow a$, where
\[
\left\Vert f\right\Vert _{\infty}=\sup\nolimits_{x\in\lbrack0,1]}\left\vert
f(x)\right\vert .
\]

\end{theorem}

\begin{proof}
$(a)$ Consider any sequence $\left\{  \left(  \lambda_{n}(t_{k}),t_{k}\right)
:k\in\mathbb{N}\right\}  $ such that $t_{k}\in(a-\delta,a+\delta)$ for all
$k\in\mathbb{N}$ and $t_{k}\rightarrow a$ as $k\rightarrow\infty,$ where
$\delta$ is defined in Theorem 1. Let $\left(  \lambda,a\right)  $ be any
limit point of the sequence $\left\{  \left(  \lambda_{n}(t_{k}),t_{k}\right)
:k\in\mathbb{N}\right\}  .$ Since $\Delta$ is a continuous function with
respect to the pair $(\lambda,t)$ and $\Delta\left(  \lambda_{n}(t_{k}%
),t_{k}\right)  =0$ for all $k,$ we have $\Delta(\lambda,a)=0.$ This means
that $\lambda$ is an eigenvalue of $L_{a}$ lying in $\left(  \Lambda
_{j}(a)-r,\Lambda_{j}(a)+r\right)  $. Hence, by Theorem 1$(b),$ we have
$\lambda=\lambda_{s_{j-1}+1}(a)=$ $\lambda_{s_{j-1}+2}(a)=...=\lambda_{s_{j}%
}(a),$ where $n\in\lbrack s_{j-1}+1,s_{j}].$ Thus, $\lambda_{n}(t_{k}%
)\rightarrow\lambda_{n}(a)$ as $k\rightarrow\infty$ for any sequence $\left\{
t_{k}:k\in\mathbb{N}\right\}  $ converging to $a$ and $\lambda_{n}$ is
continuous at $a$.

$(b)$ Using (10)-(12) we obtain the following estimation
\[
\left\vert P_{t}f(x)-P_{a}f(x)\right\vert =\left\vert \int\nolimits_{D_{j}%
(a)}\int_{0}^{1}\left(  G(x,\xi,\lambda,t)-G(x,\xi,\lambda,a)\right)
f(\xi)d\xi d\lambda\right\vert \leq
\]%
\[
2\pi rM\left\vert t-a\right\vert \int_{0}^{1}\left\vert f(\xi)\right\vert d\xi
\]
for all $x\in\lbrack0,1].$ This estimation implies the proof of $(b).$
\end{proof}

Note that in [9], the continuity of the band function for the case $\nu=1$ was
proved by using the perturbation theory from [3]. In [10], we investigated the
differential operator $T$, generated in the space $L_{2}^{m}(\mathbb{R}^{d})$
by formally self-adjoint differential expression of order $2\nu$ with matrix
coefficients, whose entries are periodic with respect to the lattice $\Omega$,
where $d\geq1$. Note that the band functions $\lambda_{1}(t)\leq\lambda
_{2}(t)\leq\cdot\cdot\cdot\ $\ and Bloch functions $\Psi_{1,t},\Psi_{2,t},...$
of $T$ are the eigenvalues and normalized eigenfunctions of the operator
$T_{t}$ generated in $L_{2}^{m}(F)$ by by the same differential expression and
the quasiperiodic conditions%
\[
u(x+\omega)=e^{i\left\langle t,\omega\right\rangle }u(x),\ \forall\omega
\in\Omega,
\]
where $t\in F^{\star}$, $\left\langle \cdot,\cdot\right\rangle $ is the inner
product in $\mathbb{R}^{d}$, $F$ and $F^{\star}$ are the fundamental domains
of the lattice $\Omega$ and dual lattice $\Gamma$, respectively. It was
proved, in [10], that the Bloch eigenvalues and corresponding projections of
the differential operator $T_{t}$ depend continuously on $t\in F^{\ast}$.
Moreover, if $\lambda_{n}(a)$\ is a simple eigenvalue, then the eigenvalues
$\lambda_{n}(t)$ are simple in some neighborhood of $a$ and the corresponding
eigenfunctions $\Psi_{n,t}$ can be chosen so that
\[
\left\Vert \Psi_{n,t}-\Psi_{n,a}\right\Vert \rightarrow0
\]
as $t\rightarrow a$. In [10], the Bloch function $\Psi_{n,t}$ was chosen so
that
\begin{equation}
\arg(\Psi_{n,t},\Psi_{n,a})=0 \tag{13}%
\end{equation}
which is not a constructive choice.

Now, instead of (13), we constructively define the normalized eigenfunctions
$\Psi_{n,t}$ that depend continuously on $t.$ First of all, let us note the
following obvious statement. If $\lambda_{n}(t)$ is a simple eigenvalue, then
the set of all normalized eigenfunctions corresponding to $\lambda_{n}(t)$ is
$\left\{  e^{i\alpha}\Psi_{n,t}:\alpha\in\lbrack0,2\pi)\right\}  ,$ where
$\Psi_{n,t}$ is a fixed normalized eigenfunction. If the eigenvalue
$\lambda_{n}(a)$ is simple, then there exists a neighborhood $U(a)$ of the
point $a\in F^{\star}$ such that for $t\in U(a)$ the eigenvalue $\lambda
_{n}(t)$ is also simple and the equality
\begin{equation}
\int_{D}(T_{t}-\lambda I)^{-1}e^{i\left\langle a,x\right\rangle }e_{k}%
d\lambda=(e^{i\left\langle a,x\right\rangle }e_{k},e^{i\alpha}\Psi
_{n,t})e^{i\alpha}\Psi_{n,t}=(e^{i\left\langle a,x\right\rangle }e_{k}%
,\Psi_{n,t})\Psi_{n,t} \tag{14}%
\end{equation}
is true for any choice of the normalized eigenfunction $\Psi_{n,t},$ where $D$
is a closed curve enclosing only the eigenvalue $\lambda_{n}(t),$ and
$e_{1},e_{2},...,e_{m}$ is the standard basis of $\mathbb{C}^{m}.$ Since the
projection operator onto the subspace corresponding to the eigenvalue
$\lambda_{n}(t)$ \ depends continuously on $t\in U(a),$ and the\ norm is a
continuous function, it follows from (14) that $\left\vert (\Psi
_{n,t},e^{i\left\langle a,x\right\rangle }e_{k})\right\vert $ is also a
continuous function with respect to $t$ in $U(a)$ for any normalized
eigenfunction $\Psi_{n,t}.$ This and the inequality
\begin{equation}
\left\vert \left\vert (\Psi_{n,t},e^{i\left\langle t,x\right\rangle }%
e_{k})\right\vert -\left\vert (\Psi_{n,t},e^{i\left\langle a,x\right\rangle
}e_{k})\right\vert \right\vert \leq\left\Vert e^{i\left\langle
t,x\right\rangle }-e^{i\left\langle a,x\right\rangle }\right\Vert \tag{15}%
\end{equation}
give the following obvious statement.

\begin{proposition}
If $\lambda_{n}(a)$ is a simple eigenvalue, then the function $\left\vert
(\Psi_{n,t},e^{i\left\langle t,x\right\rangle }e_{k})\right\vert $ does not
depend on choice of the normalized eigenfunction $\Psi_{n,t}$ and is
continuous in some neighborhood $U(a)$ of $a$, where $U(a)\subset F^{\ast}$
and any set $E$ satisfying the conditions:

$(a)$ $\left\{  \gamma+t:t\in E,\text{ }\gamma\in\Gamma\right\}
=\mathbb{R}^{d}$ and

$(b)$\ if $t\in E,$ then $\gamma+t\notin E$ for any $\gamma\in\Gamma
\backslash\left\{  0\right\}  ,$

can be used as the fundamental domain $F^{\star}$ of $\Gamma.$
\end{proposition}

Since $\left\{  e^{i\left\langle \gamma+t,x\right\rangle }e_{k}:\gamma
\in\Gamma,\text{ }k=1,2,...,m\right\}  $ is an orthonormal basis of $L_{2}%
^{m}(F)$ there exist $\gamma\in\Gamma,$ $k\in\left\{  1,2,...,m\right\}  $ and
$\varepsilon>0$ such that%
\begin{equation}
\left\vert (\Psi_{n,a},e^{i\left\langle \gamma+a,x\right\rangle }%
e_{k})\right\vert >\varepsilon. \tag{16}%
\end{equation}
For example if $\left\vert (\Psi_{n,a},e^{i\left\langle \gamma
+a,x\right\rangle }e_{k})\right\vert =\max_{\beta\in\Gamma}\left\vert
(\Psi_{n,a},e^{i\left\langle \beta+a,x\right\rangle }e_{k})\right\vert ,$ then
(16) holds. Note that if $F^{\ast}$ is a fundamental domain of $\Gamma,$ then
$\gamma+F^{\ast}$ for any $\gamma\in\Gamma$ and even $b+F^{\ast}$ for any
$b\in\mathbb{R}^{d}$ is a fundamental domain of the lattice $\Gamma.$
Therefore, without loss of generality and for the simplicity of the notation,
we will use $a$ instead of $\gamma+a$ and assume that $a$ is an interior point
of $F^{\ast}.$ Therefore, by Proposition 1 and (16) there exists a
neighborhood $U(a)$ of $a$ such that
\begin{equation}
\left\vert (\Psi_{n,t},e^{i\left\langle t,x\right\rangle }e_{k})\right\vert
>\varepsilon\tag{17}%
\end{equation}
for all $t\in U(a)$ and the normalized eigenfunction $\Psi_{n,t}$ can be
chosen so that
\begin{equation}
\arg(\Psi_{n,t},e^{i\left\langle t,x\right\rangle }e_{k})=0. \tag{18}%
\end{equation}
Then $(\Psi_{n,t},e^{i(t,x)}e_{k})=\left\vert (\Psi_{n,t},e^{i(t,x)}%
e_{k})\right\vert ,$ and hence by Proposition 1, $(\Psi_{n,t},e^{i\left\langle
t,x\right\rangle }e_{k})$ depends continuously on $t$ in some neighborhood of
$a.$ From this, taking into account (15) and (17), it follows that%
\[
\frac{(\Psi_{n,t},e^{i\left\langle a,x\right\rangle }e_{k})}{(\Psi
_{n,a},e^{i\left\langle a,x\right\rangle }e_{k})}=:\alpha(t)\rightarrow1
\]
as $t\rightarrow a.$ Therefore, using the continuity of the right side of (14)
and then (17) we obtain
\[
\left\Vert (e^{i\left\langle a,x\right\rangle }e_{k},\Psi_{n,t})\Psi
_{n,t}-(e^{i\left\langle a,x\right\rangle }e_{k},\Psi_{n,a})\Psi
_{n,a}\right\Vert \rightarrow0
\]
and $\left\Vert \alpha(t))\Psi_{n,t}-\Psi_{n,a}\right\Vert \rightarrow0$ as
$t\rightarrow a.$ Thus, we have
\begin{equation}
\left\Vert \Psi_{n,t}-\Psi_{n,a}\right\Vert \leq\left\Vert (1-\alpha
(t))\Psi_{n,t}\right\Vert +\left\Vert \alpha(t))\Psi_{n,t}-\Psi_{n,a}%
\right\Vert \rightarrow0 \tag{19}%
\end{equation}
as $t\rightarrow a.$ In other words, the following statement is proved.

\begin{proposition}
If $\lambda_{n}(a)$ is a simple eigenvalue and (17) holds, then the normalized
eigenfunction $\Psi_{n,t}$ satisfying (18) depends continuously on $t$ in
$U(a).$
\end{proposition}

\begin{remark}
Note that the constructive choice (18) is also used in [7, 8]. Namely, in [8]
for the case when the right side of (1) is equal to $-\Delta+q$ and $m=1$ (for
the Schr\"{o}dinger operator) in the neighborhood of the sphere $\left\{
t\in\mathbb{R}^{d}:\left\vert t\right\vert =\rho\right\}  ,$ where $\rho$ is a
large number, I constructed a set $B$ such that if $t\in B,$ then there exists
a unique eigenvalue $\lambda_{n(t)}(t)$ that is simple and close to
$\left\vert t\right\vert ^{2}$ and the corresponding normalized eigenfunction
$\Psi_{n,t}$ satisfies the asymptotic formula
\begin{equation}
\left\vert (\Psi_{n,t},e^{i\left\langle t,x\right\rangle })\right\vert
^{2}=1+O(\rho^{-\delta})>\tfrac{1}{2}\tag{20}%
\end{equation}
for some $\delta>0.$ Moreover, the normalized eigenfunction was chosen so that
(18) holds (see [8], p. 55). In [8] the choice (18) was made in order to write
(20) in the elegant form $\Psi_{n,t}=e^{i\left\langle t,x\right\rangle
}+O(\rho^{-\delta})$. However, in Proposition 2 we show that the choice (18)
ensures the continuity of $\Psi_{n,t}.$ Note that Proposition 2 is also new
for the Schr\"{o}dinger operator. However, Proposition 1 for the
Schr\"{o}dinger operator is obvious, since it follows directly from the
continuity of the function $e^{i\left\langle t,x\right\rangle },$ the
projection operator, and the norm. Proposition 1 and (20) were used in [8] to
prove that $n(t)=n(a)$ for all $t\in U(a)$ (see (5.11) of [8]), where
$U(a)\subset B$ and the condition $(b)$ of Proposition 1 holds (see Lemma
5.1$(b)$ of [8]), i.e., $U(a)\subset\left(  B\cap F^{\star}\right)  $ for some
fundamental domain $F^{\ast}.$
\end{remark}

Now let us return to the study of $L.$

\begin{theorem}
If $\lambda_{n}(a)$ is a simple eigenvalue, then there exists $\beta>0$ such
that the eigenvalues $\lambda_{n}(t)$ for $\left\vert t-a\right\vert <\beta$
are also simple eigenvalues and the normalized eigenfunctions $\Psi_{n,t}$ of
$L_{t}$ satisfying (17) and (18) for $\left\langle t,x\right\rangle =tx$
converges to $\Psi_{n,a}(x)$ uniformly with respect to $x\in\lbrack0,1]$ as
$t\rightarrow a$.
\end{theorem}

\begin{proof}
If $\lambda_{n}(a)$ is a simple eigenvalue, then $\frac{d\Delta(\lambda
,a)}{d\lambda}\neq0$ for $\lambda=\lambda_{n}(a).$ Then by Summary 1, Theorem
2$(b)$ and (14) there exists $\beta>0$ such that $\lambda_{n}(t)$ is also a
simple eigenvalue for $\left\vert t-a\right\vert <\beta$ and
\begin{equation}
\left\Vert (e^{iax}e_{k},\Psi_{n,t})\Psi_{n,t}-(e^{iax}e_{k},\Psi_{n,a}%
)\Psi_{n,a}\right\Vert _{\infty}\rightarrow0 \tag{21}%
\end{equation}
as $t\rightarrow a.$ Moreover, from (21) and (17) it follows that there exist
$\varepsilon>0$ and $M$ such that $\left\vert \Psi_{n,t}(x)\right\vert \leq M$
for all $\left\vert t-a\right\vert <\varepsilon$ and $x\in\lbrack0,1].$
Therefore, replacing $L_{2}$ norm $\left\Vert \cdot\right\Vert $ everywhere by
the norm $\left\Vert \cdot\right\Vert _{\infty}$ and repeating the proof of
(19) we obtain the proof of the theorem.
\end{proof}

\section{On the spectrum of $L$}

\ \ In this section, we study the spectrum of $L$. For this we consider the
operator $L_{t}(\varepsilon,C)$ generated by the differential expression
\[
L_{t}(\varepsilon,C)y=(-i)^{2\nu}y^{(2\nu)}+Cy^{(2\nu-2)}+\varepsilon\left(
(P_{2}-C)y^{(2\nu-2)}+%
{\textstyle\sum\limits_{l=3}^{2\nu}}
P_{l}(x)y^{(2\nu-l)}\right)
\]
and boundary conditions (2), where $\varepsilon\in\lbrack0,1],$ and $C$ is
defined in (5). We consider the operator $L_{t}(\varepsilon,C)$ as
perturbation of $L_{t}(C)$ by $L_{t}(\varepsilon,C)-L_{t}(C),$ where
$L_{t}(C)$ is the operator generated by the expression
\begin{equation}
(-i)^{2\nu}y^{(2\nu)}(x)+Cy^{(2\nu-2)}(x) \tag{22}%
\end{equation}
and boundary condition (2). Therefore, first of all, let us analyze the
eigenvalues and eigenfunction of the operator $L_{t}(C)$. We assume that $C$
is the Hermitian matrix. Then $L_{t}(C)$ is the self-adjoint operator, since
the expression (22) and boundary conditions (2) are self-adjoint. The distinct
eigenvalues of $C$ are denoted by $\mu_{1}<\mu_{2}<...<\mu_{p}$. If the
multiplicity of $\mu_{j}$ is $m_{j},$ then $m_{1}+m_{2}+...+m_{p}=m$. Let
$u_{j,1},$ $u_{j,2},...,u_{j,m_{j}}$ be the normalized eigenvectors of the
matrix $C$ corresponding to the eigenvalue $\mu_{j}.$ The functions
$\Phi_{k,j,s,t}(x)=u_{j,s}e^{i\left(  2\pi k+t\right)  x}$ for
$s=1,2,...,m_{j}$ are the eigenfunctions of $L_{t}(C)$ corresponding to the
eigenvalue $\ $%
\begin{equation}
\mu_{k,j}(t)=\left(  2\pi k+t\right)  ^{2\nu}+\mu_{j}\left(  2\pi k+t\right)
^{2\nu-2}, \tag{23}%
\end{equation}
since%
\begin{equation}
L_{t}(C)\Phi_{k,j,s,t}(x)=\mu_{k,j}(t)\Phi_{k,j,s,t}(x). \tag{24}%
\end{equation}

Now we consider the large eigenvalues of $L_{t}(\varepsilon,C)$. In the
forthcoming inequalities we denote by $c_{1},$ $c_{2},...$ the positive
constants that do not depend on\textit{ }$t\in(-\pi,\pi]$\textit{ }and\textit{
}$\varepsilon\in\lbrack0,1]$.

\begin{theorem}
\textit{ There exists a positive number }$N$ such that \textit{the eigenvalues
of }$L_{t}(\varepsilon,C)$\textit{ lying in }$\left(  \mu_{N,1}(t)-\varepsilon
_{N},\infty\right)  $ \textit{lie in }$\varepsilon_{k}$\textit{ neighborhood
}$U_{\varepsilon_{k}}(\mu_{k,j}(t)):=(\mu_{k,j}(t)-\varepsilon_{k},\mu
_{k,j}(t)+\varepsilon_{k})$ \textit{of }$\mu_{k,j}(t)$\textit{ for
}$\left\vert k\right\vert \geq N$ and $j=1,2,...,p,$\textit{ where
}$\varepsilon_{k}=c_{1}\left(  \mid k^{-1}\ln|k|\mid+q_{k}\right)  \left(
2\pi k\right)  ^{2\nu-2}$ and\textit{ }%
\[
q_{k}=\max\left\{  \left\vert \int\nolimits_{[0,1]}p_{2,s,r}\left(  x\right)
e^{-2\pi inx}dx\right\vert :s,r=1,2,...,m;\text{ }n=\pm2k,\pm(2k+1)\right\}
.
\]
\textit{ Moreover, for each }$\left\vert k\right\vert \geq N$ and
$j=1,2,...,p$\textit{, there exists an eigenvalue of }$L_{t}(\varepsilon
,C)$\textit{ lying in }$U_{\varepsilon_{k}}(\mu_{k,j}(t))$\textit{.}
\end{theorem}

\begin{proof}
Let $\lambda$ be the eigenvalue of $L_{t}(\varepsilon,C)$\textit{ lying in
}$\left(  \mu_{N,1}(t)-\varepsilon_{N},\infty\right)  $ and $\mu_{k,j}(t)$ be
an eigenvalue of $L_{t}(C)$ closest to $\lambda.$ We prove that $\lambda\in
U_{\varepsilon_{k}}(\mu_{k,j}(t)).$ For this we use the formula
\begin{equation}
(\lambda-\mu_{k,j}(t))(\Psi_{k,j,s,t},\Phi)=\varepsilon\left(  ((P_{2}%
-C)\Psi_{k,j,s,t}^{(2\nu-2)},\Phi)+%
{\textstyle\sum\limits_{l=3}^{2\nu}}
(P_{\nu}\Psi_{k,j,s,t}^{2\nu-l},\Phi)\right)  \tag{25}%
\end{equation}
which can be obtained from $L_{t}(\varepsilon,C)\Psi=\lambda\Psi$ by
multiplying both sides by $\Phi_{k,j,s,t}(x)$ and using (24), where $\Psi$ is
a normalized eigenfunction of $L_{t}(\varepsilon,C)$ corresponding to the
eigenvalue $\lambda.$ It was proved in [6] that there exists $c_{2}$ such
that
\begin{equation}
\left\vert ((P_{2}-C)\Psi_{k,j,s,t}^{(2\nu-2)},\Phi)+%
{\textstyle\sum\limits_{l=3}^{2\nu}}
(P_{l}\Psi_{k,j,s,t}^{2\nu-l},\Phi)\right\vert \leq c_{2}\left(  \mid\frac
{\ln|k|}{k}\mid+q_{k}\right)  \left(  2\pi k\right)  ^{2\nu-2} \tag{26}%
\end{equation}
for $\left\vert k\right\vert \geq N$ (see (51) and (54) of [6]). Moreover, by
Lemma 4 of [6], for each eigenfunction $\Psi_{k,j,s,t}$ of $L_{t}(C)$ such
that $\left\vert k\right\vert \geq N$ there exists an eigenfunction $\Phi$ of
$L_{t}(\varepsilon,C)$ satisfying
\begin{equation}
\left\vert \left(  \Psi_{k,j,s,t},\Phi\right)  \right\vert >c_{3} \tag{27}%
\end{equation}
and conversely for each eigenfunction $\Phi$ corresponding to the eigenvalue
of $L_{t}(\varepsilon,C)$\textit{ }lying in\textit{ }$\left(  \mu
_{N,1}(t)-\varepsilon_{N},\infty\right)  $ there exists $\Psi_{k,j,s,t}$
satisfying (27). Therefore, using (26) and (27) in (25) we get the proof of
the theorem.
\end{proof}

Now, using Theorem 4 and repeating the proof of Theorem 2.3, Corollary 2.4,
and Theorem 2.5 of [9] we obtain the following theorem about the bands and gaps.

\begin{theorem}
$(a)$ There exists a positive integer $N_{1}$ such that if $s\geq N_{1}$ then
the interval $\left[  a(s),b(s)\right]  $ is contained in each of the bands
$I_{sm+1},I_{sm+2},...,I_{sm+m},$ where
\[
a(s)=(s\pi)^{2v}+\mu_{p}\left(  \pi s\right)  ^{2\nu-2}+\varepsilon(s),\text{
}b(s)=\left(  s\pi+\pi\right)  ^{2}+\mu_{1}\left(  s\pi+\pi\right)  ^{2\nu
-2}-\varepsilon(s),
\]
$I_{n}$ is defined in (4), $\varepsilon(s)=\varepsilon_{k}$ if $s\in\left\{
2k,2k+1\right\}  $ and $\varepsilon_{k}$ is defined in Theorem 4.

$(b)$ Let $(\alpha,\beta)$ be the spectral gap of $L$ such that $\alpha
>b(N_{1}).$ Then $(\alpha,\beta)$ is contained in the interval
$U(s):=(b(s),a(s+1))$ for some $s\geq N_{1}$. Moreover, the spectral gap
$(\alpha,\beta)\subset U(s)$ lies between the bands $I_{sm+m}(Q)$ and
$I_{sm+m+1}(Q)$ and its length does not exceed $2\max\left\{  \varepsilon
(s),\varepsilon(s+1)\right\}  .$
\end{theorem}

For a detailed study of $\sigma(L),$ by using the asymptotic formulas,\ we
need to consider the multiplicities of the eigenvalues of $L_{t}(C)$ and the
exceptional points of the spectrum of $L(C).$ The multiplicity of $\mu
_{k,j}(t)$ is $m_{j}$ if $\mu_{k,j}(t)\neq\mu_{n,i}(t)$ for all $(n,i)\neq
(k,j).$ The multiplicity of $\mu_{k,j}(t)$ is changed, that is, $\mu_{k,j}(t)$
is an exceptional point of $\sigma(L(C))$ if $\mu_{k,j}(t)=\mu_{n,i}(t)$ for
some $(n,i)\neq(k,j).$ To consider the exceptional points of $\sigma(L(C))$
and $\sigma(L)$ we use the notation\ $a_{k}\asymp b_{k}$ which means that
there exist constants $c_{4},$ $c_{5},$ $c_{6}$ such that $c_{4}%
|a_{k}|<\left\vert b_{k}\right\vert <c_{5}|a_{k}|$ for all $\left\vert
k\right\vert >c_{6}.$ It follows from (23) that if $t\in\lbrack-\frac{\pi}%
{2},\frac{3\pi}{2}),$ then $\mu_{k,j}(t)-\mu_{k,i}(t)\asymp k^{2\nu-2}$ for
$j\neq i$ and $\left\vert \mu_{k,j}(t)-\mu_{n,i}(t)\right\vert \geq d_{k}$ for
$n\neq k,-k,-(k+1),$ where $d_{k}\asymp k^{2\nu-1}.$ Thus, the large
eigenvalue $\mu_{k,j}(t)$ for $t\in\lbrack-\frac{\pi}{2},\frac{3\pi}{2})$ may
become an exceptional Bloch eigenvalue of $L(C)$ if at least one of the
following equalities holds
\begin{equation}
\mu_{k,j}(t)=\mu_{-k,i}(t),\text{ }\mu_{k,j}(t)=\mu_{-k-1,i}(t).\tag{28}%
\end{equation}
Therefore we need to consider the points $t\in\lbrack-\frac{\pi}{2},\frac
{3\pi}{2})$ for which the equalities in (28) do not hold. Moreover, to prove
that the eigenvalues of $L_{t}$ lying in $\varepsilon_{k}=o(k^{2\nu-2})$
neighborhood of $\mu_{k,j}(t)$ (see Theorem 4) do not coincide with the
eigenvalues lying in $\varepsilon_{-k}$ and $\varepsilon_{-k-1}$ neighborhood
of $\mu_{-k,i}(t)$ and $\mu_{-k-1,i}(t)$ we consider the points $t\in
\lbrack-\frac{\pi}{2},\frac{3\pi}{2})$ for which
\begin{equation}
\left\vert f(t)\right\vert >\varepsilon_{k}+\varepsilon_{-k}\text{,
}\left\vert g(t)\right\vert >\varepsilon_{k}+\varepsilon_{-k-1}\text{,}%
\tag{29}%
\end{equation}
where $f(t)=\mu_{k,j}(t)-\mu_{-k,i}(t),$ $g(t)=\mu_{k,j}(t)-\mu_{-k-1,i}(t).$
Using (23) and the binomial expansion of $(a+b)^{n}$ for $n=2\nu$ and
$n=2\nu-2$ we obtain
\[
f(t)=(2\pi k)^{2\nu-2}(8\nu k\pi t+\mu_{j}-\mu_{i})+O(k^{2\nu-3}),\text{
}f\left(  \frac{\mu_{i}-\mu_{j}}{8\nu k\pi}\right)  =O(k^{2\nu-3}).
\]
On the other hand, one can easily verify that $f^{^{\prime}}(t)\asymp
k^{2\nu-1}.$ Therefore, there exists $\delta_{k}=o(k^{-1})$ such that the
first inequality of (29) holds if $t$ does not belong to the interval
\[
\left(  \frac{\mu_{i}-\mu_{j}}{8\nu k\pi}-\delta_{k},\frac{\mu_{i}-\mu_{j}%
}{8\nu k\pi}+\delta_{k}\right)  .
\]
In the same way we prove that if $t$ does not belong to the interval
\[
\left(  \pi+\frac{\mu_{i}-\mu_{j}}{4\pi\nu(2k+2\nu-1)}-\delta_{k},\pi
+\frac{\mu_{i}-\mu_{j}}{4\pi\nu(2k+2\nu-1)}+\delta_{k}\right)  ,
\]
then the second inequality of (29) holds. Therefore, using (29) and Theorem 4
and repeating the proof of Corollary 2.8 \ and Theorem 2.10 of [9] we obtain.

\begin{theorem}
$(a)$ There exist $N_{2}>N_{1}$ and $\gamma_{k}=o(k^{2\nu-2})$ such that the
spectral gap $(\alpha,\beta)$ defined in Theorem 5 and lying in $U(k)$ for
$k>N_{2}$ is contained in the intersection of the sets
$S(1,k),S(2,k),...,S(p,k),$ where
\[
S(j,k)=%
{\textstyle\bigcup\limits_{i=1,2,...,p}}
\left(  \left(  \pi k\right)  ^{2\nu}+\frac{\mu_{i}+\mu_{j}}{2}\left(  \pi
k\right)  ^{2\nu-2}-\gamma_{k},\left(  \pi k\right)  ^{2s}+\frac{\mu_{i}%
+\mu_{j}}{2}\left(  \pi k\right)  ^{2\nu-2}+\gamma_{k}\right)  .
\]

$(b)$ If there exists a triple $(j_{1},j_{2},j_{3})$ such that
\begin{equation}
\min_{i_{1},i_{2},i_{3}}\left(  diam(\{\mu_{j_{1}}+\mu_{i_{1}},\mu_{j_{2}}%
+\mu_{i_{2}},\mu_{j_{3}}+\mu_{i_{3}}\})\right)  \neq0,\tag{30}%
\end{equation}
where minimum is taken under condition $i_{s}\in\left\{  1,2,...,p\right\}  $
for $s=1,2,3$ and
\[
diam(E)=\sup_{x,y\in E}\mid x-y\mid,
\]
then there exists a number $H$ such that $(H,\infty)\subset\sigma(L)$ and the
number of the gaps in $\sigma(L)$ is finite.
\end{theorem}

\end{document}